\documentclass[reqno,12pt,a4paper]{amsart}

\usepackage[utf8]{inputenc}

\usepackage{enumerate}
\usepackage{amstext,amsmath,amsthm,amsfonts,amssymb,amscd}
\usepackage{latexsym,mathrsfs,dsfont,euscript}
\usepackage{fullpage}
\usepackage{xcolor}

\usepackage{hyperref} 
\hypersetup{
  colorlinks,%
  citecolor=blue,%
    filecolor=red,%
    linkcolor=red,%
    urlcolor=blue
}

\newtheorem{theorem}{Theorem}[section]
\newtheorem{proposition}[theorem]{Proposition}
\newtheorem{lemma}[theorem]{Lemma}
\newtheorem{corollary}[theorem]{Corollary}
\theoremstyle{definition}

\theoremstyle{remark}

\numberwithin{equation}{section}

\DeclareMathOperator{\supp}{supp}
\newcommand{\abs}[1]{\left\vert#1\right\vert}
\newcommand{\set}[1]{\left\{#1\right\}}

\newcommand{\norm}[1]{\left\Vert#1\right\Vert}

\allowdisplaybreaks

\begin{document}
\title[]{Approximating the identity of convolution with random mean and random variance}
%


\author[]{Hugo Aimar}
\author[]{Ivana G\'{o}mez}




\begin{abstract}
We provide sufficient conditions on the profile $\varphi$, on the sequence of random variables $\varepsilon_j>0$ and on the sequence of random vectors $y_j\in\mathbb{R}^n$ such that $\mathscr{E}\left(\frac{1}{\varepsilon_j^n(\omega)}\int_{z\in\mathbb{R}^n}\varphi\left(\frac{\abs{x-z-y_j(\omega)}}{\varepsilon_j(\omega)}\right)f(z) dz\right)\underset{j\to\infty}{\longrightarrow} f(x)
$ for almost every $x\in\mathbb{R}^n$, $f\in L^p(\mathbb{R}^n)$, $1\leq p\leq\infty$, where $\mathscr{E}$ denotes the expectation, $\varepsilon_j$ tends to $0\in\mathbb{R}$ in law and $y_j$ tends to $\mathbf{0}\in\mathbb{R}^n$ in law.
\end{abstract}
\maketitle

\section{Introduction}

Classical Harmonic Analysis is strongly related to the two basic geometric operations of translation and dilation. A paradigmatic analytic problem involving these basic operators is the pointwise approximation of the identity of convolution. The basic question is the almost everywhere convergence of $K_\varepsilon * f(x) = \int_{y\in\mathbb{R}^n}K_\varepsilon (x-y)f(y) dy$ to $f(x)$ when $\varepsilon\to 0^+$, with $K_\varepsilon(x)=\tfrac{1}{\varepsilon^n}K(\tfrac{x-y}{\varepsilon})$, $\int_{\mathbb{R}^n} K(x)dx=1$ and $f$ belongs to $L^p(\mathbb{R}^n)$, $1\leq p\leq \infty$. The most general conditions on $K$ in order to solve the almost everywhere convergence problem goes back to the 1976 results due to Felipe Z\'{o} (see \cite{Zo76} or \cite{deGuzmanBook}).
It is worthy noticing at this point that the main result in \cite{Zo76} can also be obtained from the vector valued approach to the Calder\'on-Zygmund theory of singular integrals. The classical results (see \cite{SteinBook70}) of decreasing kernels are covered by this general approach.

In order to introduce the problem that we consider in this note, let us recall that the basic property of a nonnegative kernel $K_\varepsilon(x,y)=\tfrac{1}{\varepsilon^n}K(\tfrac{x-y}{\varepsilon})$ in order to produce an approximation to the Dirac delta, is that $\int_{x\in\mathbb{R}^n}K(x) dx=1$. In other words $K$ and hence each $K_\varepsilon(\cdot,y)$ are probability densities in $\mathbb{R}^n$. Assume that $K$ has finite second moments, i.e. $\int_{x\in\mathbb{R}^n}\abs{x}^2K(x) dx<\infty$. An important case of the above is that of $K(x)=\varphi(\abs{x})$. The symmetry of the kernel shows immediately that the expected value of the probability distribution $K_\varepsilon(\cdot,y)$ is $y$. Moreover, its variance is $\varepsilon^2$. Hence $\varepsilon$ is its standard deviation.

In this note we aim to consider the pointwise convergence of the mean when the standard deviation $\varepsilon$ and the expected value $y$ of $K_\varepsilon(\cdot - y)$ are random variables that converge ``in law'' to the corresponding Dirac deltas at the origins of $\mathbb{R}$ and of $\mathbb{R}^n$ respectively.

Let us precise the above. Let $K(x)=\varphi(\abs{x})$. Let $(\Omega,\mathscr{F},\mathscr{P})$ be a probability space. Let $\varepsilon_j:\Omega\to\mathbb{R}^+$ be a sequence of random variables distributed according to the one dimensional probability measures $\nu_j$. Let $y_j:\Omega\to\mathbb{R}^n$ be a sequence of random vectors distributed according to the $n$ dimensional probability measures $\mu_j$. Assume that $\nu_j\to\delta_0$ and $\mu_j\to\delta_{\mathbf{0}}$ for $j\to\infty$ vaguely where $\delta_0$ and $\delta_{\mathbf{0}}$ denote the Dirac deltas at $0\in\mathbb{R}$ and at $\mathbf{0}\in\mathbb{R}^n$ respectively. Set $\Pi_j$ to denote the joint distribution of $\varepsilon_j$ and $y_j$. We consider sufficient conditions on $\varphi$, $\nu_j$, $\mu_j$ and $\Pi_j$ in order to have
\begin{equation*}
\mathscr{E}\left(\frac{1}{\varepsilon_j^n(\omega)}\int_{z\in\mathbb{R}^n}\varphi\left(\frac{\abs{x-z-y_j(\omega)}}{\varepsilon_j(\omega)}\right)f(z) dz\right)\underset{j\to\infty}{\longrightarrow} f(x)
\end{equation*}
for almost every $x\in\mathbb{R}^n$, $f\in L^p(\mathbb{R}^n)$, $1\leq p\leq\infty$. Here $\mathscr{E}(X)$ denotes the expected value of the random variable $X$. We may write the above expression for the expectation using translation and mollification operators $(\tau_y g)(x)=g(x-y)$ and $M_\varepsilon g(x)=\tfrac{1}{\varepsilon^n}g(\tfrac{x}{\varepsilon})$, as $\mathscr{E}(\tau_{y_j(\omega)}M_{\varepsilon_j(\omega)}K * f(x))$ with $K(x)=\varphi(\abs{x})$. Which at least formally can be written as a convolution operator in $\mathbb{R}^n$ with kernel $\int_{\Omega}\tau_{y_j(\omega)}M_{\varepsilon_j(\omega)}K(x) d\mathscr{P}(\omega)$. In general, this new kernel is not of mollification type and the general results in \cite{Zo76} are useful to deal with them. In other words, the Calder\'on-Zygmund decomposition methods apply under some conditions on $\varphi$ and on the distributions of $\varepsilon_j$ and $y_j$.

The paper is organized as follows. Section~\ref{sec:PointwiseConvergence} is devoted to the pointwise convergence for continuous functions and to the pointwise convergence for general functions under the assumption of the weak type (1,1) of the maximal operators. In Section~\ref{sec:MaximalOperator} we deal with the estimates and boundedness properties of the maximal operator under different assumptions on $\varphi$, $\nu_j$, $\mu_j$ and $\Pi_j$ the joint distribution of $\varepsilon_j(\omega)$ and $y_j(\omega)$.

\section{Pointwise convergence}\label{sec:PointwiseConvergence}
Given a probability space $(\Omega,\mathscr{P})$, a sequence of random variables $X_j$ defined in $\Omega$ is said to converge in distribution, in law, or weakly to the random variable $X$ if the distribution $\mu_j=\mathscr{P}\circ X^{-1}_j$ converge vaguely to $\mu=\mathscr{P}\circ X^{-1}$, the distribution of $X$. The vague convergence coincides with the convergence of $\int g(x)d\mu_j(x)$ to $\int g(x)d\mu(x)$ for every bounded and continuous function $g(x)$. For real valued random variables, their distributions are probability measures on $\mathbb{R}$. In this case the vague convergence is equivalent to the convergence of $\mu_j((a,b])$ to $\mu((a,b])$ for $a$ and $b$ in a dense subset of $\mathbb{R}$. (See \cite{ChungBook}). In particular, a sequence $\varepsilon_j$ of random variables tends to $0$ in law when the distribution $\nu_j$ of $\varepsilon_j$ tends to $\delta_0$ vaguely. And the sequence of random vectors $y_j$ tends to $\mathbf{0}\in\mathbb{R}^n$ in law when $\mu_j$, the $n$ dimensional distribution sequence, tends to $\delta_{\mathbf{0}}$ vaguely.

Let us precise the above. For $j=1,2,\ldots$, let $\varepsilon_j:\Omega\to\mathbb{R}^+=\{\varepsilon>0\}$ be a sequence of positive random variables. Let $\nu_j(B)=\mathscr{P}(\varepsilon_j^{-1}(B))$, $B$ a Borel set in $\mathbb{R}$, be the distribution of $\varepsilon_j$. The convergence in law of $\varepsilon_j$ to $0$ is equivalent to the weak or vague convergence of $\nu_j$ to $\delta_0$, the Dirac unit mass at the origin $0\in\mathbb{R}$. Since the distribution $\mu_j$ of $y_j$ is given by $\mu_j(B)=\mathscr{P}(y_j^{-1}(B))$ for each Borel subset $B$ of $\mathbb{R}^n$, the convergence in law of $y_j$ to $\mathbf{0}\in\mathbb{R}^n$ is equivalent to the weak or vague convergence of $\mu_j$ to $\delta_{\mathbf{0}}$, the Dirac unit mass at the origin $\mathbf{0}\in\mathbb{R}^n$. 
The joint distribution of $\varepsilon_j$ and $y_j$ is given by $\Pi_j(E)= \mathscr{P}(\set{\omega\in\Omega: (\varepsilon_j(\omega),y_j(\omega))\in E})$ for every Borel sets $E$ of $\mathbb{R}^{n+1}$.

The integrability condition on the kernel profile $\varphi:\mathbb{R}^+\to \mathbb{R}^+\cup\set{0}$ is $\omega_n\int_0^\infty\rho^{n-1}\varphi(\rho)d\rho=1$, where $\omega_n$  is the surface measure of the unit sphere in $\mathbb{R}^n$. Set $K(x)=\varphi(\abs{x})$. Since we are dealing with the differentiation problem, the operators involved are positive, hence we can consider nonnegative measurable functions $f:\mathbb{R}^n\to\mathbb{R}$ in order to have well defined, for each $x\in\mathbb{R}^n$, the mean
\begin{equation*}
m_j(f)(x)=\mathscr{E}[(\tau_{y_j(\omega)} M_{\varepsilon_j(\omega)}K)* f(x)],
\end{equation*}
which could be infinite. Notice that in particular, if $f$ is continuous and bounded. Since for each $\omega\in\Omega$, $\varepsilon_j(\omega)>0$ we have that
\begin{equation*}
m_j(f) (x) = \mathscr{E}\left[\int_{z\in\mathbb{R}^n}\varphi(\abs{z})f(x-y_j(\omega)-\varepsilon_j(\omega)z) dz\right] \leq \norm{f}_\infty,
\end{equation*}
for every $x$.

\begin{proposition}\label{propo:MeanConvergence}
Assume that $\Pi_j$, the joint distribution of $(\varepsilon_j,y_j)$, converges vaguely to $\delta_{(0,\mathbf{0})}$, the Dirac unit mass at $(0,\mathbf{0})\in\mathbb{R}^{n+1}$. Let $g$ be a continuous and bounded function defined on $\mathbb{R}^n$. Then 
\begin{equation*}
m_j g(x) \to g(x)
\end{equation*}
for every $x\in\mathbb{R}^n$, when $j\to\infty$.
\end{proposition}
\begin{proof}
Since $\Pi_j$ is the distribution of the random vector $(\varepsilon_j,y_j)$ we have that
\begin{align*}
\abs{m_j g(x) - g(x)} &\leq \mathscr{E}\left[\int_{z\in\mathbb{R}^n}\varphi(\abs{z}) \abs{ g(x-y_j(\omega)-\varepsilon_j(\omega)z) - g(x)} dz\right]\\
&= \int_{(s,y)\in\mathbb{R}^{n+1}} \left(\int_{z\in\mathbb{R}^n}\varphi(\abs{z}) \abs{ g(x-y-sz) - g(x)} dz \right) d\Pi_j(s,y).
\end{align*}
Notice now that the function $G(s,y)=\int_{z\in\mathbb{R}^n}\varphi(\abs{z})\abs{g(x-y-sz)-g(x)}dz$ is continuous, bounded and $G(0,\mathbf{0})=0$. Hence $\abs{m_j(g)(x)-g(x)}\to 0$ as $j\to\infty$ from the weak convergence of $\Pi_j$ to $\delta_{(0,\mathbf{0})}$.
\end{proof}

The next result contains an elementary sufficient condition for the vague convergence of $\Pi_j$ to $\delta_{(0,\mathbf{0})}$.

\begin{lemma}
	Let $(\Omega,\mathscr{F},\mathscr{P})$ be a probability space. Let $\varepsilon_j(\omega)$ be a positive random variable distributed by $\nu_j$ for each $j=1,2,3,\dots$ Assume that $\nu_j$ converges vaguely to $\delta_0$, the Dirac delta at $0$. Let $y_j(\omega)$ be an $\mathbb{R}^n$ valued random vector distributed by $\mu_j$. If there exists a positive constant $C$ such that the inequality $\abs{y_j(\omega)}\leq C \varepsilon_j(\omega)$ almost surely for every $j$. Then the distribution $\Pi_j$ of the $\mathbb{R}^+\times\mathbb{R}^n$ valued random vector $Z_j(\omega)=(\varepsilon_j(\omega),y_j(\omega))$ converges vaguely to $\delta_{(0,\mathbf{0})}$, the Dirac delta in $\mathbb{R}^{n+1}$. Precisely, for any continuous and bounded function $g$ in $\mathbb{R}^{n+1}$ we have that
	\begin{equation*}
	\int_{\mathbb{R}}\idotsint_{\mathbb{R}^n} g(s,y) d\Pi_j(s,y)\longrightarrow g(0;\mathbf{0})
	\end{equation*}
	when $j$ tends to infinity.
\end{lemma}
\begin{proof}
Let $D$ be a dense subset of $\mathbb{R}$ such that $\nu_j((a,b])\to\delta_0((a,b])$ for every choice of $a$ and $b$ in $D$ with $a<b$. Hence $\nu_j((a,b])\to 1$ if $a<0\leq b$ and $\nu_j((a,b])\to 0$ if $b<0$ or $a\geq 0$. Let $h_1<0<h_2$ be two real numbers in $D$. Then
	
	\begin{align*}
	&\int_{\mathbb{R}}\idotsint_{\mathbb{R}^n} \abs{g(s,y)-g(0,\mathbf{0})} d\Pi_j(s,y)\\
	&\leq \int_{(h_1,h_2]}\idotsint_{\overline{B}(\mathbf{0},Ch_2)} \abs{g(s,y)-g(0,\mathbf{0})} d\Pi_j(s,y)
	+ 2\norm{g}_\infty \Pi_j[\mathbb{R}^{n+1}\setminus(h_1,h_2]\times \overline{B}(\mathbf{0},Ch_2)],
	\end{align*}	
where $\overline{B}(\mathbf{0},Ch_2)]$ is the closure of teh ball  $B(\mathbf{0},Ch_2)]$.
For $\abs{h_1}$ and $h_2$ small enough, the first term on the right is as small as desired uniformly in $j$ because $g$ is continuous. For the second term, since $\Pi_j$ is the joint distribution of the random vector $(\varepsilon_j,y_j)$, we have
	\begin{align*}
	\Pi_j[\mathbb{R}^{n+1}\setminus (h_1,h_2]\times \overline{B}(\mathbf{0},Ch_2)] &=
	\mathscr{P}(\{(\varepsilon_j,y_j)\notin (h_1,h_2]\times \overline{B}(\mathbf{0},Ch_2)\})\\
	&= \mathscr{P}(\{\varepsilon_j\notin (h_1,h_2]  \textrm{ or } y_j\notin \overline{B}(\mathbf{0},Ch_2)\})\\
	&\leq \mathscr{P}(\{\varepsilon_j\notin (h_1,h_2]\}),
	\end{align*}
	since $\abs{y_j} > Ch_2$ implies $\varepsilon_j > h_2$. In other words,
	\begin{equation*}
	\Pi_j[\mathbb{R}^{n+1}\setminus(h_1,h_2]\times \overline{B}(\mathbf{0},Ch_2)] \leq 1-\nu_j((h_1,h_2]))
	\end{equation*}
	which tends to zero for $j\to +\infty$.
\end{proof}

In order to concentrate in Section~\ref{sec:MaximalOperator} the more technical aspects regarding the boundedness of the maximal operators, let us state and sketch the proof of the convergence result assuming the weak type of the underlying maximal operator.

Let $\varphi:\mathbb{R}^+\to\mathbb{R}^+\cup \{0\}$ with $\omega_n\int_0^\infty\rho^{n-1} \varphi(\rho) d\rho =1$. Let $\varepsilon_j$ be a sequence of positive random variables and $y_j$ a sequence of $n$ dimensional random vectors. For any measurable function $f$ the sublinear and positively homogeneous maximal operator
\begin{align*}
\mathscr{M}f(x) &= \sup_{j\geq 1} \mathscr{E}[\tau_{y_j}M_{\varepsilon_j}K * \abs{f}(x)]\\&=\sup_{j\geq 1} \mathscr{E}\left[\int_{z\in\mathbb{R}^n}\frac{1}{\varepsilon_j^n(\omega)}\varphi\left(\frac{\abs{x-y_j(\omega)-z}}{\varepsilon_j(\omega)}\right)\abs{f(z)}dz\right]
\end{align*}
is well defined. Of course with the above general conditions it could be identically equal to $+\infty$.

As usual we say that $\mathscr{M}$ is of weak type (1,1) if there exists a constant $A>0$ such that
\begin{equation*}
\abs{\set{x\in\mathbb{R}^n: \mathscr{M} f(x)>\lambda}} \leq\frac{A}{\lambda}\norm{f}_{L^1(\mathbb{R}^n)}
\end{equation*}
for every $\lambda >0$. Here the vertical bars denote the Lebesgue measure in $\mathbb{R}^n$.
For the sake of completeness we state and sketch the proof of the convergence theorem under the assumption of weak type (1,1) of $\mathscr{M}$.
\begin{theorem}
Let $\varphi\geq 0$ with $\int_{\mathbb{R}^n}\varphi(\abs{x})dx=1$. Assume that the joint distribution $\Pi_j$ of $\varepsilon_j$ and $y_j$, $j\in\mathbb{Z}^+$, converges weakly to $\delta_{(0,\mathbf{0})}$ and that $\mathscr{M}$ is of weak type (1,1). Then,  for every $f\in L^1(\mathbb{R}^n)$ we have that
\begin{equation*}
m_j(f)(x)\underset{j\to\infty}{\longrightarrow} f(x)
\end{equation*}
for almost every $x\in\mathbb{R}^n$.
\end{theorem}
\begin{proof}
For $f\in L^1(\mathbb{R}^n)$ and $g$ any continuous and compactly supported function on $\mathbb{R}^n$, from Proposition~\ref{propo:MeanConvergence} we have that the exceptional set of convergence for $f$ is the same as the exceptional set of convergence of $f-g$, hence
	\begin{align*}
	\left\vert\left\{\right.\right. x\in\mathbb{R}^n:& \,\mathscr{E}[(\varphi_{\varepsilon_j}* f)(x+y_j)]\nrightarrow f(x)\left.\right\}\left.\right\vert\\
	&= \abs{\set{x\in\mathbb{R}^n: \mathscr{E}[(\varphi_{\varepsilon_j}* (f-g))(x+y_j)]\nrightarrow (f-g)(x)}}\\
	&= \abs{\set{x\in\mathbb{R}^n: \limsup_j\abs{\mathscr{E}[(\varphi_{\varepsilon_j}* (f-g))(x+y_j)] - (f-g)(x)}>0}}\\
	&= \abs{\bigcup_{k\geq 1}\set{x\in\mathbb{R}^n: \limsup_j\abs{\mathscr{E}(\varphi_{\varepsilon_j}* (f-g))(x+y_j) - (f-g)(x)}>\frac{1}{k}}}.
	\end{align*}
	On the other hand, for each $k=1,2,3,\ldots$ from the weak type of $\mathscr{M}$ and Chebyshev inequality
	\begin{align*}
	&\abs{\set{x\in\mathbb{R}^n: \limsup_j\abs{\mathscr{E}(\varphi_{\varepsilon_j}* (f-g))(x+y_j) - (f-g)(x)}>\frac{1}{k}}}\\
	&\leq \abs{\set{x\in\mathbb{R}^n: \mathscr{M} (f-g)(x)>\frac{1}{2k}}} +
	\abs{\set{x\in\mathbb{R}^n: \abs{f(x)-g(x)}>\frac{1}{2k}}}\\
	&\leq 2k (A+1) \norm{f-g}_{L^1(\mathbb{R}^n)},
	\end{align*}
	since the continuous functions are dense in $L^1(\mathbb{R}^n)$ we have that each set
	\begin{equation*}
	\set{x\in\mathbb{R}^n:\limsup_j \abs{\mathscr{E}[(\varphi_{\varepsilon_j}* (f-g))(x+y_j)-(f-g)(x)]}>\frac{1}{k}}
	\end{equation*}
	has Lebesgue measure equal to zero and the theorem is proved.
\end{proof}

\section{The maximal operator}\label{sec:MaximalOperator}
In this section we deal with the estimates leading to the weak type (1,1) inequality for $\mathscr{M}$. Since $\mathscr{M}$ is bounded in $L^\infty(\mathbb{R}^n)$ Marcinckiewicz interpolation will provide the $L^p(\mathbb{R}^n)$ boundedness of $\mathscr{M}$ also for $1<p\leq\infty$. 


The next result contains some basic properties of the kernel of the operators $m_j(f)$.
\begin{proposition}\label{propo:propertieskerneloperators}
Let $\varphi:\mathbb{R}^+\cup\{0\}\to\mathbb{R}^+$ with $\int_0^\infty\rho^{n-1}\varphi(\rho) d\rho<\infty$. Let $\Pi_j$ be the distribution measure of $(\varepsilon_j,y_j)$. Then,
\begin{enumerate}[(\ref{propo:propertieskerneloperators}.a)]
	\item (random variance and random mean) $m_j(f)(x)=\int_{z\in\mathbb{R}^n}K_j(x-z)f(z) dz$, with
	\begin{equation*}
	K_j(x) = \int_{\mathbb{R}^+}\int_{\mathbb{R}^{n}}\frac{1}{s^n}\varphi\left(\frac{\abs{x-y}}{s}\right) d\Pi_j(s,y);
	\end{equation*}
	\item (random variance only) if $y_j\equiv 0$ for every $j\in\mathbb{Z}^+$, $K_j(x)=\psi_j(\abs{x})$ with $\psi(t)=\int_{\mathbb{R}^+}\frac{1}{s^n}\psi(\frac{t}{s}) d\nu_j(s)$;
	\item (random mean only) if $\varepsilon_j\equiv 0$ for every $j\in\mathbb{Z}^+$, $m_j(f)(x)=\int_{y\in \mathbb{R}^n} f(x-y) d\mu_j(y)$;
	\item (the case of self-similarity of $\Pi_j$) if $\Pi_j(E)=\Pi(jE)$ for every $j\in\mathbb{Z}^+$ and every Borel set $E$ in $\mathbb{R}^{n+1}$, $K_j(x)=j^n K(jx)$ with $K(x)=\int_{\mathbb{R}^+}\int_{\mathbb{R}^{n}}\frac{1}{s^n}\varphi(\frac{\abs{x-y}}{s}) d\Pi(s,y)$.
	\end{enumerate}
\end{proposition}
\begin{proof}
In order to prove \textit{(\ref{propo:propertieskerneloperators}.a)} we only have to compute the expectation of the random variable $\frac{1}{\varepsilon_j^n(\omega)}\varphi\left(\frac{\abs{x-y_j(\omega)}}{\varepsilon_j(\omega)}\right)$ in term of the distribution $\Pi_j$ of the $n+1$ dimensional random vector $(\varepsilon_j(\omega), y_j(\omega))$. The formula in \textit{(\ref{propo:propertieskerneloperators}.b)} follows from \textit{(\ref{propo:propertieskerneloperators}.a)} with $\Pi_j=\nu_j\times \delta_{\mathbf{0}}$. The formula in \textit{(\ref{propo:propertieskerneloperators}.c)} follows directly from the definition of $m_j$.

Let us prove \textit{(\ref{propo:propertieskerneloperators}.d)}. Notice first that the identity $\Pi_j(E)=\Pi(jE)$ can be written as $\iint_{\mathbb{R}^{n+1}}\mathcal{X}_E(s,y) d\Pi_j(s,y)=\iint_{\mathbb{R}^{n+1}}\mathcal{X}_E(\frac{s}{j},\frac{y}{j}) d\Pi(s,y)$. From  this identity and standard arguments we have that $\iint_{\mathbb{R}^{n+1}}g(s,y) d\Pi_j(s,y)=\iint_{\mathbb{R}^{n+1}}g(\frac{s}{j},\frac{y}{j}) d\Pi(s,y)$, for $g\geq 0$ measurable. Taking, for fixed $x\in\mathbb{R}^n$, $g(s,y)=\frac{1}{s^n}\varphi(\frac{\abs{x-y}}{s})$ we have
	\begin{align*}
	K_j(x) &= \iint_{\mathbb{R}^{n+1}}\frac{1}{s^n}\varphi\left(\frac{\abs{x-y}}{s}\right) d\Pi_j(s,y)\\
	&= \iint_{\mathbb{R}^{n+1}}\left(\frac{j}{s}\right)^n\varphi\left(\frac{\abs{x-\frac{y}{j}}}{\frac{s}{j}}\right) d\Pi(s,y)\\
	&= j^n\iint_{\mathbb{R}^{n+1}}\frac{1}{s^n}\varphi\left(\frac{\abs{jx-y}}{s}\right) d\Pi(s,y)\\
	&= j^n K(jx).
	\end{align*}
\end{proof}

Since of the main results in \cite{Zo76} are the central tools of our analysis, we shall briefly recall them here.

The operator $\sup_{\alpha\in\Lambda}\abs{K_\alpha * f}$ is of weak type (1,1) if the integrals $ \int_{\mathbb{R}^n}\abs{K_\alpha(x)} dx$ are uniformly bounded for $\alpha\in\Lambda$ and $\int_{\abs{x}\geq 2\abs{z}}\sup_{\alpha\in\Lambda}\abs{K_\alpha(x-z)-K_\alpha(x)}dx$ are uniformly bounded in $z\in\mathbb{R}^n$. The classical conditions for $K_\lambda(x)=\lambda^n K(\lambda x)$ with $K$ positive and integrable, is given in terms of the size of the gradient of $K$; $\abs{\nabla K(x)}\leq\frac{C}{\abs{x}^{n+1}}$. On the other hand, the result contained in Theorem~4 in \cite{Zo76} provides an integrable function $f$ for each non lacunary sequence $\lambda_j$ such that
$\limsup_{j\to\infty}K_{\lambda_j}(f)(x)=+\infty$ almost everywhere when $K$ is unbounded and increasing in the interval $[0,1]$. In particular, singularity of the kernel outside the origin provide unbounded maximal operators. 

Let us start with the case of random variance (only) with probability measures $\Pi_j=\nu_j\times \delta_{\mathbf{0}}$ with $\nu_j$ self-similar, $j\in\mathbb{Z}^+$. 
\begin{theorem}\label{thm:CaseRVonlyselfsimilarity}
Let $\varphi:\mathbb{R}^+\to\mathbb{R}^+\cup \{0\}$ be a $\mathscr{C}^1$ function with $\int_{0}^{\infty}\rho^{n-1}\varphi(\rho) d\rho<\infty$ and $\abs{\varphi'(\rho)}\leq \frac{B}{\rho^{n+1}}$ for some constant $B$ and every $\rho$ positive. Let $\varepsilon_j$ be a sequence of random variables distributed according $\nu_j$. Assume that $\{\nu_j: j\in\mathbb{Z}^+\}$ are self-similar. Then the maximal operator
\begin{equation*}
\mathscr{M}f(x) =\sup_{j\geq 0} m_j(\abs{f})(x)
\end{equation*}
is of weak type (1,1).
\end{theorem}
\begin{proof}
The self-similarity of $\nu_j$ implies the self-similarity  $\Pi_j=\nu_j\times\delta_{\mathbf{0}}$. 
Hence from \textit{(\ref{propo:propertieskerneloperators}.d)}, $K_j(x)=j^nK_1(jx)$. In order to apply the gradient criteria stated above, we are led to prove the integrability of $K_1$, 
\begin{align*}
\int_{\mathbb{R}^n}\abs{K_1(x)}dx &= \int_{\mathbb{R}^n}\int_{\mathbb{R}^+} \frac{1}{s^n}\varphi\left(\frac{\abs{x}}{s}\right) d\nu_1(s) dx\\&= \omega_n\int_0^\infty \rho^{n-1}\int_{\mathbb{R}^+}\frac{1}{s^n}\varphi\left(\frac{\rho}{s}\right) d\nu_1(s) d\rho\\
&=\omega_n \int_{\mathbb{R}^+}\frac{1}{s^n}\left(\int_0^\infty \rho^{n-1}\varphi\left(\frac{\rho}{s}\right)d\rho\right) d\nu_1(s)\\ 
&= \omega_n\int_0^\infty\rho^{n-1}\varphi(\rho) d\rho
\end{align*}
which is finite. Let us check that the gradient of $K_1$ has the desired size estimate. Let $i=1,2,\ldots,n$, for $\abs{x}\neq 0$ we can take the $i$-th partial derivative of $K_1(x)$ inside the integral. Hence
\begin{equation*}
\frac{\partial K_1}{\partial x_i}(x) = \int_{\mathbb{R}^+}\frac{1}{s^n}\varphi'\left(\frac{\abs{x}}{s}\right)\frac{1}{s}\frac{x_i}{\abs{x}} d\nu_1(s).
\end{equation*}
Then 
\begin{equation*}
\abs{\nabla K_1(x)}\leq B\int_{\mathbb{R}^+}\frac{1}{s^n}\left(\frac{s}{\abs{x}}\right)^{n+1}\frac{1}{s} d\nu(s)=\frac{B}{\abs{x}^{n+1}}
\end{equation*}
as desired.
\end{proof}

The next result shows that the hypothesis of self-similarity of $\nu_j$ in Theorem~\ref{thm:CaseRVonlyselfsimilarity} above can be avoided.
\begin{theorem}\label{thm:nujNoSelfSimilarity}
Let $\varphi$ be as in Theorem~\ref{thm:CaseRVonlyselfsimilarity}. Let $\varepsilon_j(\omega)$ be a positive random variable distributed according to $\nu_j$ and $y_j\equiv 0$, $j\in\mathbb{Z}^+$. Then the maximal operator $\mathscr{M}$ is of weak type (1,1).
\end{theorem}
\begin{proof}
Let us show that
\begin{enumerate}[(a)]
	\item $\int_{\mathbb{R}^n}\abs{K_j(x)} dx$ are uniformly bounded in $j$; and
	\item $\int_{\abs{x}\geq 2\abs{z}}\sup_{j\geq 1}\abs{K_j(x-z)-K_j(x)} dx\leq A<\infty$ uniformly in $z\in\mathbb{R}^n$.
\end{enumerate}
Recall that $K_j(x)=\int_{\mathbb{R}^+}\frac{1}{s^n}\varphi\left(\frac{\abs{x}}{s}\right)d\nu_j(s)$. Let us check (a),
\begin{align*}
\int_{\mathbb{R}^n}\abs{K_j(x)} dx &= \int_{\mathbb{R}^+}\frac{1}{s^n}\int_{\mathbb{R}^n}\varphi\left(\frac{\abs{x}}{s}\right)dx d\nu_j(s)\\
&= \int_{\mathbb{R}^n}\varphi(\abs{x})dx\\
&= \omega_n\int_0^\infty\rho^{n-1}\varphi(\rho) d\rho. 
\end{align*}
In order to prove (b), fix $z\in\mathbb{R}^n$ with $\abs{z}>0$. Notice that for $\abs{x}\geq 2\abs{z}$ we have that  any $\xi$ in the segment joining $x-z$ and $x$ has length larger than or equal to $\tfrac{\abs{x}}{2}$. In fact, set $\xi=\theta(x-z)+(1-\theta)x$, with $0<\theta<1$, then 
\begin{align*}
\abs{x}&=\abs{\theta x + (1-\theta)x}\leq\abs{\theta(x-z)+(1-\theta)x}+\theta\abs{z}
=\abs{\xi}+\theta\abs{z}\\
&\leq\abs{\xi}+\abs{z}\leq\abs{\xi}+\frac{\abs{x}}{2}.
\end{align*}
Hence for $\abs{x}\geq 2\abs{z}$ we have
\begin{align*}
\abs{K_j(x-z)-K_j(x)} &\leq \int_{\mathbb{R}^+}\frac{1}{s^n}\abs{\varphi\left(\frac{\abs{x-z}}{s}\right)-\varphi\left(\frac{\abs{x}}{s}\right)} d\nu_j(s)\\&\leq\int_{\mathbb{R}^+}\frac{1}{s^n}\abs{\varphi'\left(\frac{\abs{\xi(x,z)}}{s}\right)} \frac{\abs{z}}{s} d\nu_j(s) \\ &\leq B\abs{z} \int_{\mathbb{R}^+}\frac{1}{s^{n+1}}\left(\frac{s}{\abs{\xi(x,z)}}\right)^{n+1} d\nu_j(s)\\
&\leq 2^n B\frac{\abs{z}}{\abs{x}^{n+1}}.
\end{align*}
Since $\int_{\abs{x}\geq 2\abs{z}}\frac{\abs{z}}{\abs{x}^{n+1}} dx =\omega_n\abs{z}\int_{2\abs{z}}^\infty\frac{d\rho}{\rho^2}=\frac{\omega_n}{2}$ we have (b) and the proof of the theorem.
\end{proof}

Let us now consider some particular instances of non identically vanishing means $y_j(\omega)$.

Let us start from the case \textit{(\ref{propo:propertieskerneloperators}.c)} in Proposition~\ref{propo:propertieskerneloperators}. In this case the operator $m_j(f)$ is given by $m_j(f)(x)=\int_{\mathbb{R}^n}f(x-y) d\mu_j(y)$ where $\mu_j=\mathscr{P}\circ y_j^{-1}$ is the $n$ dimensional distribution of the random means $y_j$. Of course that there is a large class of cases of sequences $\set{\mu_j: j\in\mathbb{Z}^+}$ that produce a well behaved maximal operator. But in general the behavior of the maximal operator in this case is far away from being good. At this point Theorem~4 in \cite{Zo76} helps to build self-similar probability measures such that the differentiation $m_j(f)(x)\to f(x)$ a.e., fails. For the sake of completeness let us rephrase Theorem~4 in \cite{Zo76} in our current notation.
\begin{theorem}\label{thm:Theorem4Zocurrent}
Let $n=1$. Assume that $\mu_j(B)=\mu_1(jB)$; $j\in\mathbb{Z}^+$ with $d\mu_1(y)=\gamma(y) dy$ with $\gamma\in L^1(\mathbb{R})$, $\gamma$ unbounded and nondecreasing in $(0,1)$. Then, there exists an integrable function $f$ in $\mathbb{R}$ such that
\begin{equation*}
\limsup_{j\to\infty} m_j(f)(x)= +\infty
\end{equation*}
almost everywhere.
\end{theorem}
The above results show that in some sense the random character of the variance is less restrictive that the random character of the mean for the differentiation properties of $\mathscr{E}[\tau_{y_j} M_{\varepsilon_j}K * f]$. Let us consider now some cases where both $\{\varepsilon_j\}$ and $\{y_j\}$ are nontrivial. We start by providing some sufficient conditions on the joint distribution of $\varepsilon_j$ and $y_j$ for the weak type of $\mathscr{M}$ when the profile $\varphi(t)=\mathcal{X}_{(0,1)}(t)$ is the standard for the Lebesgue differentiation theorem through centered Euclidean balls.
\begin{theorem}\label{thm:SufficientConditionsonJointDistributionforWeakType}
Let  $\varphi(\rho)=\mathcal{X}_{(0,1)}(\rho)$. Assume that $\Pi_j(E)=\Pi_1(jE)$, $j\in\mathbb{Z}^+$ that $\supp\Pi_1\subset [0,\alpha]\times \overline{B}(0,\beta)$ and that $d\Pi_1(s,y)=\gamma(s,y) dsdy$ with $\gamma\in L^1(\mathbb{R}^+,L^\infty(\mathbb{R}^n))$. Then $\mathscr{M}(f)$ is dominated by the standard Hardy-Littlewood maximal function $f^*$ defined on the centered balls of $\mathbb{R}^n$.
\end{theorem}
\begin{proof}
From \textit{(\ref{propo:propertieskerneloperators}.d)} in Proposition~\ref{propo:propertieskerneloperators} we have to estimate the kernel $K_1$. From \textit{(\ref{propo:propertieskerneloperators}.a)}
\begin{align*}
K_1(x) &= \int_{\mathbb{R}^+}\int_{\mathbb{R}^n}\frac{1}{s^n}\varphi\left(\frac{\abs{x-y}}{s}\right) d\Pi_1(s,y)\\
&= \int_{\mathbb{R}^+}\int_{\mathbb{R}^n}\frac{1}{s^n}\mathcal{X}_{(0,1)}\left(\frac{\abs{x-y}}{s}\right) d\Pi_1(s,y)\\
&= \int_{\mathbb{R}^+}\frac{1}{s^n}\int_{B(x,s)}\gamma(s,y)dy ds\\
&= \int_0^\alpha\frac{1}{s^n}\int_{B(x,s)\cap B(0,\beta)}\gamma(s,y)dy ds.
\end{align*}
So that $K_1(x)\leq V_n\norm{\gamma}_{L^1(\mathbb{R}^+,L^\infty(\mathbb{R}^n))}$ with $V_n$ the volumen of unit ball of $\mathbb{R}^n$. On the other hand, for $\abs{x}>\alpha + \beta$ we have that $B(x,s)\cap B(0,\beta)=\emptyset$ and $K_1(x)=0$. Hence $K_1(x)\leq \norm{\gamma}_{L^1(\mathbb{R}^+,L^\infty(\mathbb{R}^n))}\mathcal{X}_{(0,\alpha+\beta)}(\abs{x})$ and  the result is proved.
\end{proof}

Again, the self-similarity of the measures $\Pi_j$ can be substituted by more general conditions. Notice that under the hypotheses of Theorem~\ref{thm:SufficientConditionsonJointDistributionforWeakType} we have that $d\Pi_j(s,y)=\gamma_j(s,y)dsdy$ with $\gamma_j(s,y)=j^{n+1}\gamma(js,jy)$. So that the boundedness of the density $\gamma$ implies that $\norm{\gamma_j}_\infty=j^{n+1}\norm{\gamma}_\infty$ and $\supp\Pi\subseteq [0,\alpha]\times \overline{B}(0,\beta)$ implies  $\supp\Pi_j\subseteq [0,\tfrac{\alpha}{j}]\times \overline{B}(0,\tfrac{\beta}{j})$.
\begin{theorem}\label{thm:BoundednessgammaImpliesPointwiseSupremum}
Let $\varphi=\mathcal{X}_{(0,1)}$. Assume that $d\Pi_j=\gamma_j ds dy$, $\supp \gamma_j\subseteq [0,r_j]\times B(0,r_j)$ and $\norm{\gamma_j}_\infty\leq\frac{A}{r_j^{n+1}}$. Then $\mathscr{M}(f)(x)$ is bounded by a constant time $f^*(x)$ for every $x\in\mathbb{R}^n$.
\end{theorem}
\begin{proof}
Since $K_j(x)=\int_0^{r_j}\frac{1}{s^n}\int_{B(x,s)\cap B(0,r_j)}\gamma_j(s,y)dyds$, we see that if $\abs{x}\geq 2 r_j$, then $K_j(x)=0$. That is, $\supp K_j\subseteq B(0,2r_j)$. On the other hand,
\begin{equation*}
K_j(x) \leq \norm{\gamma_j}_\infty\int_0^{r_j}\frac{1}{s^n}\abs{B(x,s)\cap B(0,r_j)} ds
\leq V_n\frac{A}{r_j^{n+1}}r_j
=\frac{V_n A}{r_j^n},
\end{equation*}
for every $x\in\mathbb{R}^n$. Hence
\begin{equation*}
\sup_{j\geq 1} (K_j * f)(x) \leq V_n A\sup_{j\geq 1} \frac{1}{r_j^n}\int_{B(x,2r_j)}\abs{f(z)} dz,
\end{equation*}
and $\mathscr{M}f(x)\leq 2^n A V_n^2 f^*(x)$.
\end{proof}
Let us finally provide sufficient conditions on $\varepsilon_j$, $y_j$ in order to have the weak type of $\mathscr{M}$ when $\varphi$ in not localized, now with $y_j\neq 0$. In the proof of our result we shall use the following corollary of Theorem~\ref{thm:BoundednessgammaImpliesPointwiseSupremum}.
\begin{corollary}\label{coro:ofTheoremBoundednessgammaImpliesPointwiseSupremum}
Let $\varphi=\varphi_{\alpha,\beta}=\alpha\mathcal{X}_{(0,\beta)}$ with $\alpha$ and $\beta$ positive. Assume that $\Pi_j$ satisfies the hypotheses of  Theorem~\ref{thm:BoundednessgammaImpliesPointwiseSupremum}. Then
\begin{equation*}
\mathscr{M}(f)(x)\leq \alpha\beta^n(\beta+1)^{n}A\,V_n^2 f^*(x)
\end{equation*} 
for every $x\in\mathbb{R}^n$.
\end{corollary} 
\begin{proof}
Set $K_j^{\alpha,\beta}(x)$ to denote the kernel defined by $\varphi_{\alpha,\beta}$ and $\Pi_j$. Then 
\begin{align*}
K_j^{\alpha,\beta}(x) &= \int_{\mathbb{R}^+}\int_{\mathbb{R}^n}\frac{1}{s^n}\alpha \mathcal{X}_{(0,\beta)}\left(\frac{\abs{x-y}}{s}\right) d\Pi_j(s,y)\\
&\leq  \alpha\norm{\gamma_j}_\infty\int_0^{r_j}\int_{B(x,s\beta)\cap B(0,r_j)}\frac{1}{s^n} dy ds\\
&\leq\alpha\norm{\gamma_j}_\infty\int_0^{r_j}\frac{\abs{B(x,s\beta)}}{s^n}ds\\
&= \alpha\beta^n V_n\norm{\gamma_j}_\infty r_j\\
&\leq \alpha\beta^n V_n\frac{A}{r_j^n}.
\end{align*}
On the other hand, $K_j^{\alpha,\beta}(x)$ vanishes for $\abs{x}\geq (\beta+1)r_j$. In fact, for this values of $x$ we have that $B(x,s\beta)\cap B(0,r_j)=\emptyset$. Then $K_j^{\alpha,\beta}(x)\leq \alpha\beta^n V_n\frac{A}{r_j^n}\mathcal{X}_{B(0,(\beta+1)r_j)}(x)=\alpha\beta^n(\beta+1)^{n} V_n^2A\frac{1}{\abs{B(0,(\beta+1)r_j)}}\mathcal{X}_{B(0,(\beta+1)r_j)}(x)$ and we are done.
\end{proof}

\begin{theorem}\label{thm:lastTheorem}
Let $\varphi:\mathbb{R}^+\to\mathbb{R}^+\cup\{0\}$ be bounded and non increasing with $\int_{\mathbb{R}^n}\abs{x}^n\varphi(\abs{x}) dx<\infty$. Let $d\Pi_j=\gamma_j ds dy$ with $\supp\gamma_j\subseteq [0,r_j]\times \overline{B}(0,r_j)$ and $\norm{\gamma_j}_\infty\leq\frac{A}{r_j^{n+1}}$.
Then $\mathscr{M}(f)(x)$ is dominated by the Hardy-Littlewood maximal function $f^*(x)$.
\end{theorem}
\begin{proof}
From the integrability and monotonicity properties of $\varphi$ we have
\begin{align*}
\int_{\mathbb{R}^n}\abs{x}^n\varphi(\abs{x}) dx &\geq \omega_n\int_1^\infty\rho^{2n-1}\varphi(\rho) d\rho\\&= \omega_n\sum_{l\geq 0}\int_{2^l}^{2^{l+1}} \rho^{2n}\varphi(\rho) \frac{d\rho}{\rho}\\
&\geq \omega_n\sum_{l\geq 0}\varphi(2^{l+1})2^{2ln}\int_{2^l}^{2^{l+1}} \frac{d\rho}{\rho}\\
&= \omega_n\log 2\sum_{l\geq 0}\varphi(2^{l+1})2^{2ln}.
\end{align*}
%
Hence 
Now $(K_j * \abs{f})(x) = (K_j^0 *\abs{f})(x) + \sum_{l\geq 1}(K_j^l *\abs{f})(x)$ with
\begin{equation*}
K_j^0(x) = \int_{\mathbb{R}^+}\int_{\mathbb{R}^n}\frac{1}{s^n}\varphi\left(\frac{\abs{x-y}}{s}\right) \mathcal{X}_{(0,1)}\left(\frac{\abs{x-y}}{s}\right)d\Pi_j(s,y)
\end{equation*}
and
\begin{equation*}
K_j^l(x) = \int_{\mathbb{R}^+}\int_{\mathbb{R}^n}\frac{1}{s^n}\varphi\left(\frac{\abs{x-y}}{s}\right) \mathcal{X}_{(2^{l-1},2^l)}\left(\frac{\abs{x-y}}{s}\right)d\Pi_j(s,y).
\end{equation*}
Since $\varphi$ is bounded $K_j^0 * f$ is uniformly dominated by $f^*$ because of Theorem~\ref{thm:BoundednessgammaImpliesPointwiseSupremum}. On the other hand, since for $l\geq 1$ we have
\begin{equation*}
K_j^l(x)\leq \varphi(2^{l-1})\int_{\mathbb{R}^+}\int_{\mathbb{R}^n}\frac{1}{s^n} \mathcal{X}_{(0,2^l)}\left(\frac{\abs{x-y}}{s}\right)d\Pi_j(s,y).
\end{equation*}
Then, from Corollary~\ref{coro:ofTheoremBoundednessgammaImpliesPointwiseSupremum}, with $\alpha=\varphi(2^{l-1})$ and $\beta=2^l$,
\begin{equation*}
\sup_{j\geq 1} (K_j^l * \abs{f})(x)\leq \varphi(2^{l-1})2^{nl}(2^l+1)^n V_n^2 A f^*(x)
\end{equation*}
and we are done since
\begin{align*}
\sum_{l\geq 1}\varphi(2^{l-1})2^{nl}(2^l+1)^n&\leq \varphi(1)2^n 6^n + \sum_{l\geq 0} \varphi(2^{l+1}) 2^{2ln}\\
&\leq \varphi(1)2^n 6^n +\frac{1}{\omega_n \log 2}\int_{\mathbb{R}^n}\abs{x}^n\varphi(\abs{x}) dx\\
&<\infty.
\end{align*}
\end{proof}

Notice that the conditions on $\varphi$ in Theorem~\ref{thm:nujNoSelfSimilarity} allow heavy tails for $\varphi$. And also singularity of $\varphi$ at the origin. Instead in Theorem~\ref{thm:lastTheorem} the profile $\varphi$ is bounded and heavy tails like Poisson type kernels are not covered.


\providecommand{\bysame}{\leavevmode\hbox to3em{\hrulefill}\thinspace}
\providecommand{\MR}{\relax\ifhmode\unskip\space\fi MR }
\providecommand{\MRhref}[2]{%
	\href{http://www.ams.org/mathscinet-getitem?mr=#1}{#2}
}
\providecommand{\href}[2]{#2}

\section*{Acknowledgements} 
This work was supported by the Ministerio de Ciencia, Tecnolog\'ia e Innovaci\'on-MINCYT in Argentina: Consejo Nacional de Investigaciones Cient\'ificas y T\'ecnicas-CONICET and Agencia Nacional de Promoci\'on Cient\'ifica y T\'ecnica-ANPCyT, (Grant PICT 2015-3631).


\bigskip

%

\noindent{\footnotesize
\noindent\textit{Affiliations:\,}
\textsc{Instituto de Matem\'{a}tica Aplicada del Litoral, UNL, CONICET.}

\smallskip
\noindent\textit{Address:\,}\textmd{CCT CONICET Santa Fe, Predio ``Alberto Cassano'', Colectora Ruta Nac.~168 km 0, Paraje El Pozo, S3007ABA Santa Fe, Argentina.}

\smallskip
\noindent \textit{E-mail address:\, }\verb|haimar@santafe-conicet.gov.ar|; \verb|ivanagomez@santafe-conicet.gov.ar| 
}
%

\end{document}